\documentclass[psamsfonts]{amsart}
\usepackage{amssymb,amsfonts}
\usepackage{mathrsfs}
\usepackage{setspace}
\usepackage{enumitem}
\usepackage{xcolor}
\usepackage{graphicx}
\usepackage{amsaddr}

\usepackage[hidelinks]{hyperref}
\usepackage{fancyhdr}
\usepackage[margin=1in]{geometry}

\theoremstyle{theorem}
\newtheorem{thm}{Theorem}[section]
\newtheorem{cor}[thm]{Corollary}
\newtheorem{prop}[thm]{Proposition}
\newtheorem{lemma}[thm]{Lemma}

\newtheorem{question}[thm]{Question}

\theoremstyle{definition}
\newtheorem{definition}[thm]{Definition}

\newtheorem{construction}[thm]{Construction}

\newtheorem{notation}[thm]{Notation}

\newtheorem{obsv}[thm]{Observation}
\newtheorem{fact}[thm]{Fact}

\theoremstyle{remark}

\newcommand{\R}{\mathbb{R}}

\newcommand{\vv}{\mathbf{v}}

\newcommand{\vx}{\mathbf{x}}
\newcommand{\vy}{\mathbf{y}}

\newcommand{\vq}{\mathbf{q}}
\newcommand{\vp}{\mathbf{p}}

\newcommand{\vj}{\mathbf{j}}
\newcommand{\vzr}{\mathbf{0}}

\newcommand{\xmax}{q_{\max}}
\newcommand{\xmin}{q_{\min}}
\newcommand{\vmax}{v_{\max}}
\newcommand{\vmin}{v_{\min}}
\newcommand{\ee}{\mathrm{e}}
\newcommand{\ignore}[1]{}

\DeclareMathOperator{\spec}{spec}

\DeclareMathOperator{\degag}{diag}
\DeclareMathOperator{\dist}{dist}
\DeclareMathOperator{\aut}{Aut}

\DeclareMathOperator{\qavg}{qavg}
\DeclareMathOperator{\Ring}{Ring}

\makeatletter
\let\c@equation\c@thm
\def\l@subsection{\@tocline{2}{0pt}{2pc}{6pc}{}} 
\makeatother
\numberwithin{equation}{section}

\makeatletter
\let\c@equation\c@thm
\def\l@subsubsection{\@tocline{2}{0pt}{4pc}{6pc}{}} 
\makeatother
\numberwithin{equation}{section}

\title{On the principal eigenvector of a graph}
\author{Yueheng Zhang}
\address{The University of Chicago}
\date{Jul 29, 2021}
\thanks{Written for the University of Chicago 2019 Math REU; mentored by Professor L\'{a}szl\'{o} Babai.}

\begin{document}

\begin{abstract}
The principal ratio of a connected graph $G$, $\gamma(G)$, is the ratio between the largest and smallest coordinates of the principal eigenvector of the adjacency matrix of $G$. Over all connected graphs on $n$ vertices, $\gamma(G)$ ranges from $1$ to $n^{cn}$. Moreover, $\displaystyle{\gamma(G)=1}$ if and only if $G$ is regular. This indicates that $\gamma(G)$ can be viewed as an irregularity measure of $G$, as first suggested by Tait and Tobin (El. J. Lin. Alg. 2018). We are interested in how stable this measure is. In particular, we ask how $\gamma$ changes when there is a small modification to a regular graph $G$. We show that this ratio is polynomially bounded if we remove an edge belonging to a cycle of bounded length in $G$, while the ratio can jump from $1$ to exponential if we join a pair of vertices at distance $2$. We study the connection between the spectral gap of a regular graph and the stability of its principal ratio. A naive bound shows that given a constant multiplicative spectral gap and bounded degree, the ratio remains polynomially bounded if we add or delete an edge. Using results from matrix perturbation theory, we show that given an \emph{additive} spectral gap greater than $(2+\epsilon)\sqrt{n}$, the ratio stays bounded after adding or deleting an edge.
\end{abstract}

\maketitle
\tableofcontents

\section{Introduction}

It is known that the adjacency matrix of every connected graph has a simple largest eigenvalue,  and this eigenvalue has an eigenvector with all-positive coordinates, called the principal eigenvector of the graph. It is natural to study how this eigenvector reflects the structure of the graph.  Our goal is to obtain asymptotic information as the number of vertices approaches infinity in a family of graphs.

Cioab\u{a} and Gregory~\cite{CioabaGregory2007} defined the \emph{principal ratio} of the graph $G$, $\gamma(G)=\xmax/\xmin$, to be the ratio between the largest and smallest coordinates of the principal eigenvector $\vq$. This ratio is $1$ for regular graphs, while it can grow at factorial rate (i.e., $\gamma(G)>n^{cn}$ for some positive constant $c$)~\cite{CioabaGregory2007}. Since $\gamma(G)\ge 1$ where equality holds if and only if $G$ is regular, it is natural to think of $\gamma(G)$ as a measure of the irregularity of $G$.  This view was suggested by Tait and Tobin~\cite{TaitTobin2018}. 

A basic observation is that, given a connected graph $G$ with largest eigenvalue $\lambda_1$ and diameter $D$, the principal ratio satisfies
\begin{equation}\label{lambda^D}
\gamma(G)\le\lambda_1^D.
\end{equation} 

We are interested in the stability of $\gamma$, i.e., how a slight change of $G$ influences $\gamma(G)$. In particular, given a $d$-regular graph $G$, we ask how $\gamma(G)$ changes from the constant $1$ if we add or remove one edge in $G$. (We call the resulting graphs $G+e$ and $G-e$, respectively.) We always assume the edge we remove will not disconnect $G$ (i.e.,  the edge $e$ is not a bridge), so that the principal eigenvector of $G-e$ is defined. 

In Section~\ref{bounded}, we study the cases where the edge we add to or remove from a regular graph is between vertices at bounded distance. We show the following.
\begin{itemize}
\item $\gamma(G+e)$ can jump to exponential in $n$ when the degree is bounded [Theorem~\ref{exponential}]. In our example, $e$ connects two vertices at distance $2$ in $G$.  We make use of Chebyshev polynomials in proving this result. We recall some properties of Chebyshev polynomials and show their connection with the principal eigenvectors of certain graphs in Sec.~\ref{chebyshev}.\\
It is easy to infer from~\eqref{lambda^D} that boundedness of the degree is necessary here (see Cor.~\ref{lambda_D_family} and Fact \ref{regular_diameter_bound}). 
\item If we remove an edge belonging to a cycle of bounded length in $G$, then $\gamma(G-e)$ is always polynomially bounded regardless of the degree [Theorem~\ref{G-e_bounded}].
\end{itemize}

We also study the relevance of the spectral gap to the stability of $\gamma(G)$ for regular graphs. In Section~\ref{multi_gap}, based on~\eqref{lambda^D}, we note the following.
\begin{itemize}
\item $\gamma(G\pm e)$ is always polynomially bounded in $n$ when $G$ is a bounded-degree expander graph, i.e., when the degree is bounded and the spectral gap of $G$ is bounded away from zero [Observation~\ref{expander}]. 
\end{itemize}
In Section~\ref{addi_gap_2}, we put this problem in the more general context of perturbations of matrices. By adapting theorems and proofs from Stewart and Sun's book~\cite{StewartSun1990} to our special case, we show the following.
\begin{itemize}
\item If the additive spectral gap is greater than $(2+\epsilon)\sqrt{n}$, then $\gamma(G\pm e)$ is bounded [Theorem~\ref{remote_2}]. 
\end{itemize}
This result does not follow from~\eqref{lambda^D}. Indeed, in Section~\ref{addi_gap_1} we construct graphs with degree of order $n^{(2+t)/3}$ and additive spectral gap of order $n^t$, having diameter of order $n^{(1-t)/3}$, for any constant $0<t<1$. Similar applications of matrix perturbation theory in link analysis for networks can be found in~\cite{Ng_ea2001}.

The main technical results of this paper are the first and the fourth bullet points.

\section{General preliminaries}\label{gp}

\subsection{Graphs}
By a \textit{graph} we mean what is often called a \textit{simple graph} (undirected graph with no self-loops and no parallel edges). $G$ will always denote a connected graph with $n$ vertices. We denote by $V(G)$ and $E(G)$ the set of vertices and edges of $G$, respectively. We usually identify the set of vertices with the set $[n]=\{1,2,\dots, n\}$. We write $i\sim_G j$ if vertices $i,j$ are adjacent in $G$. We denote by $N_G(j)$ the set of neighbors of $j$ in $G$. We use $\deg_G(j)$ to denote the degree of vertex $j$ in $G$. We write $\overline{G}$ for the complement of $G$. We let $\dist_G(i,j)$ denote the distance between vertices $i$ and $j$ in $G$, and $D(G):=\max_{i,j} \dist(i,j)$ the diameter of $G$.

Let $C_n$ denote the cycle with $n$ vertices. Let $P_r$ denote the path with $r$ vertices; it has $r-1$ edges. Let $K_s$ denote the clique with $s$ vertices; it has $\binom{s}{2}$ edges. Following the notation used in previous papers on this subject, we use $P_r \cdot K_s$ to denote the graph obtained by merging the vertex at one end of $P_r$ with one vertex in $K_s$. So $P_r\cdot K_s$ has $n=r+s-1$ vertices, $r-1+\binom{s}{2}$ edges, and diameter $r$. This has been called a \textit{kite graph} or a
\textit{lollipop graph}. We will call it a kite graph. 

\subsection{Matrices,  spectra}\label{spectrum_graphs}
Orthonormality in $\R^n$ refers to the standard dot product. Given a matrix~$A$, we write $a_{ij}$ for the entry on the $i$-th row and in the $j$-th column of $A$, and we write $A=(a_{ij})$.

We use $M_n(\R)$ to denote the set of $n\times n$ real matrices. We write $\vj_n$ for the all-ones vector in $\R^n$, and $J_n$ for the $n \times n$ all-ones matrix.  We omit the dimension when it is clear from context.

For an $m\times n$ real matrix $M$,  we write $\| M\|$ for the operator norm of $M$ induced by $\ell^2$ vector norm ($\|\cdot\|$),\[\| M\| = \sup_{\vx \in \R^n, \;\vx \ne \vzr} \frac{\|M\vx\|}{\|\vx\|}.\]

In the rest of Sec.~\ref{spectrum_graphs}, $A=(a_{ij})$ will always denote a real symmetric $n\times n$ matrix with eigenvalues $\lambda_1\ge\lambda_2\ge\cdots\ge\lambda_n$. 
\begin{definition}
A \textit{multiset} based on a set $A$ is a function from $A$ to positive integers,  denoted by $\displaystyle{\{\{ a^{m(a)} \mid a\in A \}\}}$. The value $m(a)$ is called the \emph{multiplicity} of the element $a$.  For example,  $\{\{ a^3,b,c^2 \}\}$. We also expand this notation to $\{\{a,a,a,b,c,c\}\}$. The order in which the elements are listed does not matter. So, for instance, 
\[\{\{ a^3,b,c^2 \}\} = \{\{ c^2,b,a^3 \}\}= \{\{a,a,a,b,c,c\}\} = \{\{a,b,c,a,a,c\}\}.\]
\end{definition}
\begin{definition}
The \textit{spectrum} of an $n\times n$ matrix $M$ is the multiset of its eigenvalues. We denote it by $\spec(M)$.
\end{definition}
\begin{fact}\label{spectrum_cycle}
The spectrum of the cycle $C_r$ is
\vspace{-0.4cm}
\begin{proof}[]
\[\left\{\left\{2\cos\left(\frac{2\pi j}{r}\right)\bigg| \: j = 0, 1, \dots, r-1\right\}\right\}.\qedhere\]
\end{proof}
\end{fact}
\begin{notation}
Let $B=(b_{ij})$ be an $m\times n$ matrix and let $M$ be a $p\times q$ matrix. The \emph{Kronecker product} of $B$ and $M$, $B\otimes M$, is the $mp\times nq$ matrix\[
\begin{pmatrix}
b_{11}M & b_{12}M & \cdots & b_{1n}M\\
b_{21}M & b_{22}M & \cdots & b_{2n}M\\
\vdots & \vdots & \vdots & \vdots\\
b_{m1}M & b_{m2}M & \cdots & b_{mn}M 
\end{pmatrix}.\]
\end{notation}
\begin{fact}\label{lex_fact}
Let $B\in M_n(\R)$ and $M\in M_m(\R)$. Let the eigenvalues of $B$ be $\displaystyle{\lambda_1\ge\cdots\ge\lambda_n}$ and let the eigenvalues of $M$ be $\displaystyle{\mu_1\ge\cdots\ge \mu_m}$. Then
\vspace{-0.4cm}
\begin{proof}[]
\[\spec(B\otimes M) = \{ \{ \lambda_i\mu_j \mid i \in[n], \: j \in[m]\} \}.\qedhere\]
\end{proof}
\end{fact}

\subsection{The adjacency operator,  principal eigenvector and eigenvalue}\label{adjacency_operator}
We write $A_G$ to denote the adjacency matrix of a graph $G$. We note that $A_G$ is a real symmetric matrix, so its eigenvalues are real. We refer to the eigenvalues  and eigenvectors of $A_G$ as the eigenvalues and eigenvectors of the graph $G$, respectively. We write $\lambda_1(G)\ge \lambda_2(G)\ge \cdots\ge \lambda_n(G)$ to denote the eigenvalues of $G$.
\begin{fact}\label{principal_vector}
For every connected graph $G$, the largest eigenvalue $\lambda_1(G)$ is simple, and it has a unique-up-to-scaling all-positive eigenvector (the principal eigenvector).\hfill\qed
\end{fact}

We write $\vq(G)$ for the principal eigenvector of $G$ scaled to have $\ell^2$ norm $1$. Let $q_i(G)$ denote the coordinate corresponding to vertex $i$ in $\vq(G)$. We write $q_{\max}(G)$ and $q_{\min}(G)$ for the maximum and minimum coordinates of $\vq(G)$, and $\vmax(G)$ and $\vmin(G)$ for corresponding vertices. Recall that the principal ratio of $G$ is defined as 
\begin{equation}
\gamma(G):=\frac{q_{\max}(G)}{q_{\min}(G)}.
\end{equation}

In the rest of Sec.~\ref{adjacency_operator}, we fix the graph $G$ and omit it from notations.
\begin{obsv}
Given $\vy=(y_1,\dots,y_n)$, the $i$-th coordinate of $A\vy$ is given by
\vspace{-0.4cm}
\begin{proof}[]
\[(A\vy)_i=\sum_{j:i\sim j}y_j. \qedhere\]
\end{proof}
\end{obsv}
\begin{cor}\label{eigenvector_equation}
$\vy=(y_1,\dots, y_n)$ is an eigenvector of $A$ to eigenvalue $\rho$ if and only if
\vspace*{-0.4cm}
\begin{proof}[]
 \[\rho y_i=\sum_{j:j\sim i}y_j.\qedhere\]
 \end{proof}
\end{cor}
Recall that $\vj$ denotes the all-ones vector.
\begin{cor}\label{regular_graph_eigenvalue_d}
$\vj$ is an eigenvector of $A$ if and only if $G$ is regular.\hfill\qed
\end{cor}
\begin{fact}\label{positive_then_principal}
If $G$ is connected and $\vy$ is an eigenvector to eigenvalue $\lambda$ with all-positive coordinates, then $\lambda$ is the largest eigenvalue of $A$.
\end{fact}
\begin{proof}
Since $\lambda_1$ is simple, any eigenvector $\vy$ to an eigenvalue other than $\lambda_1$ must be orthogonal to the principal eigenvector.
\end{proof}
\begin{fact}\label{regular_d_lambda_1}
For a connected $d$-regular graph $G$, $\lambda_1=d$.  \hfill\qed
\end{fact}
Let $\Delta$ denote $\max_{1\le i\le n}\deg(i)$, the maximum degree of $G$.
\begin{fact}\label{lambda_bound_maxdegree}
For every graph $G$, we have $\lambda_1\le \Delta$. For connected graphs, equality holds if and only if $G$ is regular.\hfill\qed
\end{fact}
We denote the quadratic mean of the degrees of vertices by 
\begin{equation}
\deg_{\qavg}:=\sqrt{\frac{\sum_{i=1}^n \deg(i)^2}{n}}.
\end{equation}
\begin{fact}\label{lambda_bound_quadratic_mean}
For every graph $G$,  we have $\lambda_1\ge \deg_{\qavg}$.
\end{fact}
\begin{proof}
Immediate using Rayleigh's principle applied to $A^2$ and the $\vj_n$ vector.
\end{proof}

\subsection{Spectral gaps}
For a graph $G$, we call the quantity $\lambda_1(G) - \lambda_2(G)$ the \emph{additive spectral gap} of $G$, and the quantity $(\lambda_1(G) - \lambda_2(G))/\lambda_1(G)$ the \emph{multiplicative spectral gap} of $G$.

We write $L_G$ for the Laplacian of $G$, defined as the $n\times n$ matrix \[L_G=\degag(\deg(1),\deg(2),\dots,\deg(n))-A_G.\] $L_G$ is positive semidefinite; $\vj$ is an eigenvector of $L_G$ to eigenvalue zero. We write $\delta(G)$ for the smallest eigenvalue of $L_G$ restricted to the orthogonal complement of $\vj$. We note that $\delta(G)=0$ if and only if $G$ is disconnected. $\delta(G)$ is the \emph{algebraic connectivity} of the graph $G$, as first defined by Fiedler~\cite{Fiedler1973}. 

For a graph $G$, let $f_G$ be the characteristic polynomial of the adjacency matrix, and $g_G$ the characteristic polynomial of the Laplacian.  
When $G$ is a $d$-regular graph, we have
 \begin{equation}f_G(t)=g_G(d-t).\end{equation} It follows that
\begin{equation}\label{gap_eq}
\delta(G)=d-\lambda_2(G)
\end{equation} 
is the additive spectral gap of $G$, and
\begin{equation}\frac{\delta(G)}{d}=1-\frac{\lambda_2(G)}{d}
\end{equation} is the multiplicative spectral gap of $G$. We shall discuss spectral gaps for regular graphs only.

In all notation, we omit the graph $G$ when it is clear from context.

\subsection{Rates of growth}
By a \emph{family of graphs}, we mean an infinite set of non-isomorphic finite graphs. 

Let $\mathcal{G} =\{ G_1, G_2,\dots\}$ be a family of graphs, and let $n_i$ be the number of vertices of $G_i$.  
We say $\gamma(\mathcal{G})$ is \emph{polynomially bounded} if $\displaystyle{\gamma(G_i)<n_i^c}$ for some constant $c$ and all sufficiently large $i$. We say $\gamma(\mathcal{G})$ is \emph{exponential} if $\displaystyle{\gamma(G_i)\ge a^{n_i}}$ for some constant $a>1$ and all sufficiently large $i$. We say $\gamma(\mathcal{G})$ has \emph{factorial growth} if $\displaystyle{\gamma(G_i)\ge n_i^{cn_i}}$ for some positive constant~$c$ and all sufficiently large $i$.
 
\subsection{Subgraphs, graph products}

\begin{fact}\label{lambda_bound_subgraph}
If $H$ is a proper subgraph of a connected graph $G$,  i.e., $H$ is obtained from $G$ by deleting at least one edge and possibly some vertices, then $\lambda_1(G)> \lambda_1(H)$.
\end{fact}
\begin{proof}
Immediate using Rayleigh's principle.
\end{proof}

\begin{notation}
Given a graph $G$ and $U\subseteq V(G)$, we denote the induced subgraph of $G$ on the set $U$ by $G[U]$.
\end{notation}

\begin{notation}
For graphs $H=(W,F)$ and $G=(V,E)$, the \textit{Cartesian product} of $H$ and $G$, denoted by $H\square G$, is the graph with the set $W\times V$ as vertices, and $(w_1,v_1)\sim (w_2, v_2)$ if and only if $w_1=w_2$ and $v_1\sim_G v_2$, or $v_1=v_2$ and $w_1\sim_H w_2$. For each $v\in V$, we call $(H\square G)[W\times \{v\}]$ the horizontal layer corresponding to $v$. For each $w \in W$, we call $(H\square G)[\{w\}\times V]$ the vertical layer corresponding to $w$. The horizontal layers are copies of $H$ and the vertical layers are copies of $G$. 
\end{notation}

\begin{notation}\label{lex}
For graphs $H=(W,F)$ and $G=(V,E)$, the \textit{lexicographic product} of $H$ and $G$, denoted by $H\circ G$, is the graph with the set $W\times V$ as vertices, and $(w_1,v_1)\sim (w_2, v_2)$ if and only if either $w_1\sim_H w_2$ or $w_1=w_2$ and $v_1\sim_G v_2$.  For each $w \in W$ we call 
$(H\circ G)[\{w\}\times V]$ the vertical layer corresponding to $w$.
\end{notation}

Recall that $J_n$ denotes the $n\times n$ all-ones matrix. 

\begin{obsv}\label{lex_obsv}
The adjacency matrix of $H\circ G$ is $A_H\otimes J_{|V(G)|} + I_{|V(H)|}\otimes A_G$.\hfill\qed
\end{obsv}\quad\\

\addtocontents{toc}{\protect\setcounter{tocdepth}{2}}

\section{Preliminary results about the principal eigenvector}\label{vp}

As previously introduced, $G$ will always denote a connected graph. We write $\lambda_1(G)$ for the largest eigenvalue of the adjacency matrix of $G$.  We use $\vq(G)$ to mean the principal eigenvector of the adjacency matrix, the all-positive eigenvector to $\lambda_1(G)$. We assume $\vq(G)$ is scaled to have $\ell^2$ norm $1$ unless otherwise stated.  We write $\gamma(G)$ for the principal ratio $q_{\max}(G)/q_{\min}(G)$ of $G$. We omit $G$ from these notations when the graph is clear from context.

\subsection{Observations and naive bounds on the ratio}\quad\\

First, we note that $\vq$ reflects the symmetries of $G$. 

\begin{fact}\label{principal_vector_symmetry}
The principal eigenvector $\vq$ is constant on orbits of $\aut(G)$,  the automorphism group of graph $G$.\hfill\qed
\end{fact}

Next we note some basic bounds on the ratio between the coordinates of $\vq$.

\begin{obsv}\label{ratio_bound_distance}
For two vertices $i,j$ in $G$, let $\dist(i,j)=k$. Then \[\frac{q_j}{q_i}\le \lambda_1^k.\]
\end{obsv}
\begin{proof}
If $\dist(i,j)=0$, then $q_j/q_i=1=\lambda_1^0$.
If $\dist(i,j)=1$, then by Corollary~\ref{eigenvector_equation} and since all $q_w$ are positive, $\lambda_1 q_i=\sum_{w:w\sim i} q_w \ge q_j$. Now, suppose $k\ge2$ and $\displaystyle{q_w/q_i\le\lambda_1^{k-1}}$ for all vertices $w$ at distance $k-1$ from $i$. We know $j$ is adjacent to at least one vertex $w$ at distance $k-1$ from $i$.
Then \[\frac{q_j}{q_i} = \frac{q_j}{q_w}\cdot\frac{q_w}{q_1}\le\lambda_1\cdot \lambda_1^{k-1}=\lambda_1^k.\qedhere\]
\end{proof}

Recall that $D$ denotes the diameter of the graph.

\begin{cor}\label{lambda_D}
For a connected graph $G$ of diameter~$D$, 
\vspace{-0.4cm}
\begin{proof}[]
\[\gamma \le \lambda_1^D \le \Delta^D \le (n-1)^D.\qedhere\]
\end{proof}
\end{cor}

\begin{cor}\label{lambda_D_family}
If $D$ is bounded for some family $\mathcal{G}$ of graphs , then $\gamma(\mathcal{G})$ is polynomially bounded in $n$.\hfill\qed
\end{cor}

Since $D$ is relevant in bounding the ratio, we introduce a bound on $D$ for regular graphs.

\begin{fact}\label{regular_diameter_bound}
Let $G$ be a connected $d$-regular graph. Then $\displaystyle{D\le \frac{3n}{d}}$.
\end{fact}
\begin{proof}
Pick $v_0, v_D$ in $G$ so that $\dist(v_0,v_D)=D$. Let $v_0, v_1,\dots, v_D$ be a shortest path from $v_0$ to $v_D$. Any $v_i, v_j$ with $|i-j|\ge 3$ cannot have any common neighbors, since otherwise the path will not be a shortest path. Thus
\[d\left\lceil \frac{D}{3}\right\rceil \le n.\] Therefore $D\le 3n/d$.
\end{proof}

We know $D(G+e)\le D(G)$ for any $e\in \overline{G}$. The following result shows that $D(G+e)$, and consequently, also $D(G-e)$ (if still connected), cannot differ from $D_G$ by more than a factor of $2$.

\begin{fact}\label{diameter_change}
For a connected graph $G$ with $D(G)=D$ and $e\in E(\overline{G})$, \[D(G+e)\ge\frac{1}{2}D(G).\]
\end{fact}
\begin{proof}\quad\\
(Notation: By $\dist_l(x,y)$, where $l$ is a path, we mean the distance between $x$ and $y$ along the path.)\\
We need to prove that there is a pair of vertices at distance at least $\frac{D}{2}$ in $G+e$.
Let $u, v\in V(G)$ be such that $\dist_G(u,v) = D$. If $\dist_{G+e}(u,v) = D$, we are done. Otherwise, let $p$ be a shortest path between $u$ and $v$ in $G$, and $q$ a shortest path between $u$ and $v$ in $G+e$.
Then $e$ must be on $q$.  Denote by $x$ the endpoint of $e$ which is closer to $u$ on $q$, and by $y$ the other endpoint. 
Pick the middle vertex $w$ of $p$ with $\dist_p(u,w)=\lceil\frac{D}{2}\rceil$. If $\dist_{G+e}(u,w)=\dist_G(u,w)=\lceil\frac{D}{2}\rceil$, we are done. Otherwise, any shortest path $r$ in $G+e$ from $u$ to $w$ must pass through edge $e$. Suppose we go along $r$ from $u$ to $w$. By the optimality of $q$, we may assume that $r$ and $q$ overlap from $u$ to $y$.
Now we look at the vertex $w'$ adjacent to $w$ on~$p$ which is closer to $u$ than to $v$. We have
\begin{equation}\label{local_p}
\dist_G(w',v)=\dist_p(w',v)=\left\lfloor\frac{D}{2}\right\rfloor+1.
\end{equation}
We claim that there is a shortest path in $G+e$ from $w'$ to $v$ that does not pass through $e$. Let $s$ be a shortest path in $G+e$ from $w'$ to $v$ that passes through $e$. By the optimality of $q$, we may assume that $s$ and $q$ overlap from $x$ to $v$.
Then by the optimality of $r$,
\[
\dist_s(x,w')+\dist_{G+e}(w',w)\ge \dist_r(x,w) = \dist_{G+e}(x,y) + \dist_r(y,w),
\]
that is, 
\[\dist_s(x,w')\ge \dist_r(y,w).\] Therefore 
\[\dist_s(w',y)= \dist_s(w',x) + \dist_{G+e}(x,y) \ge \dist_{G+e}(w',w) + \dist_r(w,y).\] Thus $s$ is equivalent to a path that does not pass through $e$ in $G+e$. As a result, a shortest path between $w'$ and $v$ in $G+e$ is also available in $G$. Then by \eqref{local_p}, \[\dist_{G+e}(w',v)=\dist_G(w',v)=\left\lfloor\frac{D}{2}\right\rfloor+1.\qedhere\]
\end{proof}

\subsection{Chebyshev polynomials and principal eigenvectors}\label{chebyshev}\quad

\textit{The Chebyshev polynomials of the first kind}, $T_n$, can be characterized by the recurrence \begin{equation}\label{c1}T_{n+1}(t)=2t\cdot T_n(t) - T_{n-1}(t),\end{equation} with initial values $T_0(t)=1$ and $T_1(t)=t$.\\
\textit{The Chebyshev polynomials of the second kind}, $U_n$, can be characterized by the same recurrence \begin{equation}U_{n+1}(t)=2t\cdot U_n(t) - U_{n-1}(t),\end{equation} with initial values $U_0(t)=1$ and $U_1(t)=2t$.

\begin{fact}\label{explicit_formulae}
When $|t|\ge 1$, the explicit formula for $T_n$ is
\begin{equation}
T_n(t)=\frac{1}{2}\Big(\Big(t-\sqrt{t^2-1}\Big)^n+\Big(t+\sqrt{t^2-1}\Big)^n\Big).
\end{equation}
and the explicit formula for $U_n$ is 
\vspace{-0.4cm}
\begin{proof}
\begin{equation}
U_n(t)=\frac{\Big(t+\sqrt{t^2-1}\Big)^{n+1}-\Big(t-\sqrt{t^2-1}\Big)^{n+1}}{2\sqrt{t^2-1}}.\qedhere
\end{equation}
\end{proof}
\end{fact}

\begin{fact}\label{U_monotone}
When $x>1$,  both $T_n(x)$ and $U_n(x)$ are strictly increasing.\hfill\qed
\end{fact}

\begin{definition}
A pendant path of length $k$ in $G$ consists of $k$ vertices such that the induced subgraph on them is a path; moreover, one vertex has degree $1$ in $G$ and $k-2$ vertices have degree $2$ in $G$. For example, in the graph $P_r\cdot K_s$, there is a pendant path of length $r$.
\end{definition}
\begin{obsv}\label{pendant_path}
Let $1,2,\dots, k$ be a pendant path in $G$ where consecutive vertices are adjacent and $\deg(1)=1$. 
Then for $1\le j\le k$, \[\frac{q_j}{q_1} = U_{j-1}\left(\frac{\lambda_1}{2}\right).\] 
\end{obsv}
\begin{proof}
By Corollary~\ref{eigenvector_equation}, $\lambda_1 q_1 = q_2$ and $q_{j+1}=\lambda_1 q_j-q_{j-1}$ for $1\le j\le n-1$. Therefore $q_j / q_1$ satisfies the initial values and recurrence relation of $U_{j-1}(\lambda_1/2)$.
\end{proof}
In fact, this occurence was observed by Tait and Tobin~\cite{TaitTobin2018} when they proved that the maximum principal ratio over all graphs of $n$ vertices is attained by a kite graph.\\

\section{Stability of the ratio: adding or removing edges in bounded distance}\label{bounded}

In this section and the subsequent sections, we are interested in the stability of the ratio with respect to a small change in the graph.

Let $G$ be a $d$-regular graph, so $\gamma(G) = 1$. As introduced before, we use $G+e$ to denote the graph obtained by adding an edge $e\in E(\overline{G})$ to $G$, and $G-e$ to denote the graph obtained by deleting an edge $e\in E(G)$ from $G$. We always assume $G-e$ is still connected. We are interested in the possible asymptotic behaviors of $\gamma(G+e)$ and $\gamma(G-e)$. 

We first make two simple observations. For the first one, we do not require $G$ to be regular.
\begin{obsv}\label{obsv_1}
If the diameter $D(G)$ is bounded, then $\gamma(G\pm e)$ is polynomially bounded in~$n$.
\end{obsv}
\begin{proof}
Fact~\ref{diameter_change} shows that $D(G\pm e)$ is also bounded, and the statement follows from Cororllary~\ref{lambda_D_family}.
\end{proof}
\begin{obsv}\label{obsv_2}
Let $G$ be a $d$-regular graph. If $d$ is linear in $n$, i.e., $d/n$ is bounded away from $0$, then $\gamma(G\pm e)$ is polynomially bounded in~$n$.
\end{obsv}
\begin{proof}
Fact~\ref{regular_diameter_bound} shows that $D(G)$ is bounded, and the statement follows from Observation~\ref{obsv_1}.
\end{proof}

\subsection{Adding an edge}\label{bounded_add}\quad

We show that if we add an edge $e$ between two vertices at distance $2$ to a connected regular graph $G$ of bounded degree, then $\gamma(G+e)$ can be exponentially large. 

\begin{thm}\label{exponential}
For any $d\ge2$ and $n>18d^3$ such that $n/(d-1)$ is an odd integer,  there exists a $d$-regular graph $G_{n,d}$ with $n$ vertices and an edge $e_{n,d}\in E(\overline{G_{n,d}})$ between two vertices at distance 2 in $G_{n,d}$, such that
\begin{equation}\gamma(G_{n,d} + e_{n,d})>\ee^{n/(18d^3)}.
\end{equation}
In particular,  when $d$ is bounded, $\gamma(G_{n,d}+e_{n,d})$ is exponential in $n$.
\end{thm}

The proof of this theorem will be based on a series of constructions. The graphs produced by Construction~\ref{final_1} and Construction~\ref{final_2} are the pair of graphs that are used in the proof.

\begin{construction}[$\Ring_{r,d,G_2}$]
Let $r\ge 0$ be a parameter. We label the vertices in $P_{2r+1}$ from one end to the other end as $p_{-r}$, $p_{-r+1}$, ..., $p_{-1}$, $p_0$, $p_1$, ..., $p_{r-1}$, $p_r$. Let $\displaystyle{G_1=P_{2r+1}\square K_{d-1}}$. We label the vertical layer corresponding to $p_i$ as $L_i$. Let $G_2$, which we will call a ``gadget,'' be a connected graph with two vertices $v_1$, $v_2$ of degree $1$ and all other vertices of degree $d$, and an automorphism that switches $v_1$ and $v_2$. We connect $v_1$ with every vertex in $L_{-r}$ and connect $v_2$ with every vertex in $L_r$. We call the graph obtained $\Ring_{r, d, G_2}$.
\end{construction}

\begin{obsv}\label{pendant_ring}
Let $X\le \aut(G_1)$ be the subgroup of automorphisms of $G_1$ that fixes the vertical layers and permutes the horizontal layers. The orbits of $X$ are the vertical layers. $X$ is isomorphic to the symmetric group $S_{d-1}$. This group extends to $G$. The automorphism of $G_1$ that switches $L_j$ and $L_{-j}$ for $0\le j\le r$ also extends to $G$. Therefore the coordinates of the principal eigenvector of $\Ring_{r,d,G_2}$ corresponding to all vertices in $L_j\cup L_{-j}$ are the same. We denote this value by $a_j$, for $0\le j\le r$.
\end{obsv}

\begin{construction}\label{final_1}[$\Ring_{r,d}$]
Now we specify the gadget $G_2$. We take two copies of $K_{d+1}$ and call them $H_1$, $H_2$. We label the vertices in $H_1$, $H_2$ from $1$ to $d+1$.  We remove the edge $\{1,2\}$ from $H_2$, the edge $\{3,4\}$ from $H_1$, and add the edges $\{1,3\}$ and $\{2,4\}$. We find two vertices $u_1$, $u_2$ in $V(H_1)$ such that $u_2\in\mathcal{O}(u_1)$ in $H_1-\{3,4\}$. We remove the edge $\{u_1,u_2\}$. Finally, we attach a dangling vertex $w_1$ to $u_1$, and a dangling vertex $w_2$ to $u_2$. $w_1$ and $w_2$ are the vertices that connect with $L_{-r}$ and $L_r$, respectively.  For this specific $G_2$, $\Ring(r,d,G_2)$ is a $d$-regular graph on $n=(2r+1)(d-1)+2+2(d+1)=2rd-2r+3d+3$ vertices. We write $\Ring(r,d)$ for this graph.
\end{construction}

\begin{construction}\label{final_2}[$\Ring_{r,d}+e$]
We add the edge $e=\{3,4\}\in E(\overline{\Ring_{r,d}})$ to $\Ring_{r,d}$.
\end{construction}
\includegraphics[width=13cm]{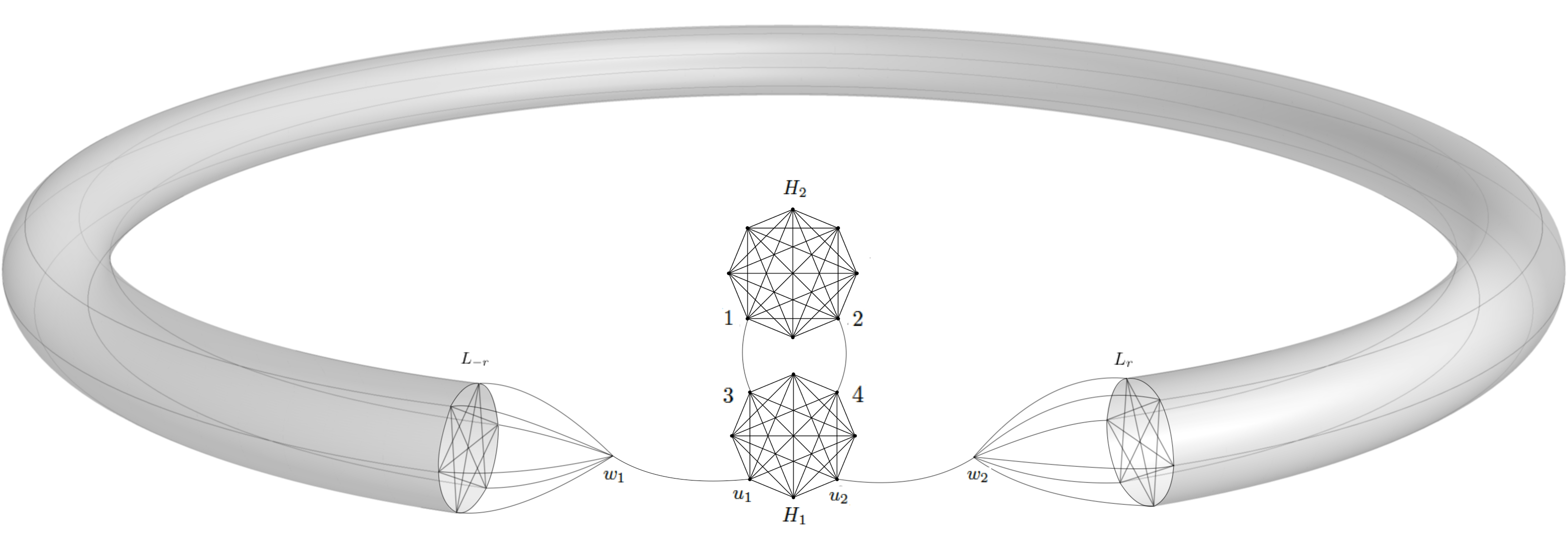}\\

\begin{prop}\label{ring}
For $0\le j \le r$, \[\frac{a_j}{a_0}=T_j\left(\frac{\lambda_1(\Ring_{r,d}+e)-d+2}{2}\right)\] where $a_j$ is as defined in Observation~\ref{pendant_ring}.
\end{prop}
\begin{proof}
By Corollary~\ref{eigenvector_equation}, \[\lambda_1 a_0 = (d-2)a_0+2a_1\] and \[\lambda_1 a_j=(d-2)a_j+a_{j-1}+a_{j+1}\] for $1\le j\le n-1$. Therefore $a_j / a_0$ satisfies the initial values and recurrence relation of $T_j((\lambda_1-d+2)/2)$ according to~\eqref{c1}.
\end{proof}

\begin{lemma}\label{lambda_jump}
$\lambda_1(\Ring_{r,d}+e)>d+c(d)$, where $c(d) := 2/(3d(d+1))$.
\end{lemma}
\begin{proof}
Let $H$ be the induced subgraph of $\Ring_{r,d}$ on $V(H_1)\cup V(H_2)$. Then \[\deg_{H+e}(1)=\deg_{H+e}(2)=d+1,\quad \deg_{H+e}(u_1)=\deg_{H+e}(u_2)=d-1,\] while the rest of the vertices in $H+e$ are of degree $d$. Then by Fact~\ref{lambda_bound_subgraph} and Fact~\ref{lambda_bound_quadratic_mean}, \[\lambda_1(\Ring_{r,d}+e)>\lambda_1(H+e)\ge \deg_{\qavg}(H+e)=\sqrt{\frac{(2d+2)d^2+4}{2d+2}}=d\sqrt{1+\frac{2}{d^2(d+1)}}.\] Since $\displaystyle{\sqrt{1+x}>1+\frac{1}{3}x}$ when $0<x<3$,  we have $\displaystyle{
\lambda_1(G+e)>d+c(d)}$, as claimed.
\end{proof}

\begin{proof}[Proof of Theorem~\ref{exponential}]
By Lemma~\ref{lambda_jump},
\[\frac{\lambda_1(\Ring_{r,d}+e)-d+2}{2} > 1+ \frac{1}{3d(d+1)}.\]
By Proposition~\ref{ring}, Fact~\ref{U_monotone}, and Fact~\ref{explicit_formulae},
\[ \gamma(\Ring_{r,d})\ge \frac{a_r}{a_0} = T_r\left(\frac{\lambda_1(\Ring_{r,d}+e)-d+2}{2}\right)
  >T_j \left(1+\frac{1}{3d(d+1)}\right) > \frac{1}{2}\left(1+\frac{1}{3d(d+1)}\right)^r.\]
Plugging in $\displaystyle{r =  \frac{n-3d-3}{2d-2}}$ and assuming $n>18d^3$,  we infer by a simple calculation that
\[\gamma(\Ring_{r,d}) >\ee^{ n/(18d^3)}.\qedhere\]
\end{proof}

\subsection{Removing an edge}\label{bounded_remove}\quad

Let $e=\{1,2\}$.

We show that in the case of removing the edge $e$ when $\dist_{G-e}(1,2)$ is bounded, $\gamma(G-e)$ is polynomially bounded for all $D$. We make use of the following theorem.

\begin{thm}[Cioab\u{a}, Gregory, Nikiforov\cite{CioabaGregoryNikiforov2007}]\label{irregular_Delta_lambda}
If $G$ is a connected nonregular graph with $n$ vertices, diameter $D$, and maximum degree $\Delta$, then 
\[\Delta - \lambda_1(G) \ge \frac{1}{n(D+1)}.\]
\end{thm}

\begin{thm}\label{G-e_bounded}
For a connected $d$-regular graph $G$ and an edge $e=\{1,2\}\in E(G)$, if $\dist_{G-e}(1,2)<c$ where $c$ is some constant, then $\gamma(G-e)$ is polynomially bounded in $n$.
\end{thm}

\begin{lemma}
$\xmin(G-e)$ is either $q_1$ or $q_2$.
\end{lemma}
\begin{proof}
If $\xmin$ corresponds to some vertex $j$ with degree $d$, then the average of the coordinates corresponding to the neighbors of $j$ would be \[\frac{\lambda_1(G-e)q_j}{d}<\frac{\lambda_1(G-e)q_j}{\lambda_1(G-e)}=q_j,\] which is a contradiction.
\end{proof}

\begin{proof}[Proof of Theorem~\ref{G-e_bounded}]
Without loss of generality suppose $q_1=\xmin$. Then summing $n$ equations of the form $\lambda_1(G-e) q_i=\sum_{j:j\sim i} q_j$, we have
\[\lambda_1(G-e)\sum_{i=1}^n q_i=(d-1) q_1 + (d-1) q_2 + dq_3 + \cdots + dq_n = d\left(\sum_{i=1}^n q_i\right) - q_1 - q_2.\]
Therefore by Theorem~\ref{irregular_Delta_lambda}, \[\frac{q_1+q_2}{\sum_{i=1}^n q_i}=d-\lambda_1(G-e)\ge\frac{1}{n(D(G-e)+1)}\ge \frac{1}{n^2}.\] By Observation~\ref{ratio_bound_distance}, $q_2\le\lambda_1(G-e)^c q_1$. Therefore \[\gamma(G-e)<\frac{\sum_{i=1}^n q_i}{q_1}\le n^2\left(1+\frac{q_2}{q_1}\right)\le n^2(1+\lambda_1(G-e)^c)<n^2(1+d^c).\qedhere\]
\end{proof}\quad\\


\section{Stability of the ratio: multiplicative spectral gap}\label{multi_gap}\quad

We use a known bound on the diameter of a graph in terms of the spectral gap to show that $\gamma(G\pm e)$ is polynomially bounded for bounded-degree expanders.

\begin{definition}
For $0<\epsilon<1$, a regular graph of degree $d$ is an $\epsilon$-expander if $\displaystyle{\lambda_2\le (1-\epsilon)d}$, i.e., $\displaystyle{\delta\ge \epsilon d}$.
\end{definition}

We note that this definition implies that an expander graph is connected.

\begin{thm}[N. Alon, V. D. Milman\cite{AlonMilman1985}]\label{diameter_Alon}
Let $G$ be a connected graph on $n$ vertices with maximum degree $\Delta$ and let $\delta$ denote the smallest positive eigenvalue of the Laplacian matrix. Then
\[
D(G)\le 2\left\lfloor\sqrt{\frac{2\Delta}{\delta}}\log_2 n\right\rfloor.\]
\end{thm}

\begin{cor}\label{expander}
For $d$-regular $\epsilon$-expander graphs $G$, we have
\begin{equation}\gamma(G+e)\le n^{4\sqrt{2/\epsilon}\log_2(d+1)}.\end{equation}
In particcular, for expander graphs $G$ with bounded degree, $\gamma(G\pm e)$ is polynomially bounded in $n$.
\end{cor}
\begin{proof}
By definition, \[\frac{\Delta(G)}{\delta(G)}=\frac{d}{\delta(G)}\le\frac{1}{\epsilon}.\]
Then \[D(G)\le 2\left\lfloor\sqrt{\frac{2}{\epsilon}} \log_2 n\right\rfloor.\]
 By Fact~\ref{diameter_change}, \[D(G\pm e) \le 4\left\lfloor\sqrt{\frac{2}{\epsilon}} \log_2 n\right\rfloor.\] 
 Since $\gamma \le \Delta^D$ (Corollary~\ref{lambda_D}),
 \begin{align*}
 \gamma(G\pm e) & \le (d+1)^{4\left\lfloor\sqrt{2/\epsilon} \log_2 n\right\rfloor}\\
 &\le (d+1)^{4\sqrt{2/\epsilon}\log_2n}\\
 &= 2^{4\sqrt{2/\epsilon}\log_2 n\log_2 (d+1)}\\
 &=n^{4\sqrt{2/\epsilon}\log_2(d+1)}.
 \qedhere
 \end{align*}
\end{proof}\quad\\

\section{Stability of the ratio: additive spectral gap}\label{addi_gap}\quad

We show that a large ($(2+\epsilon)\sqrt{n}$) additive spectral gap implies that $\gamma(G\pm e)$ is bounded. Specifically, we prove the following.
\begin{thm}\label{remote_2}
Let $G$ be a connected $d$-regular graph. If the additive spectral gap $\delta=d-\lambda_2$ of $G$ satisfies $\displaystyle{\delta>\frac{2}{c}\sqrt{n}}+2$ for some value $0<c<1$, 
then \[\gamma(G\pm e)<\frac{1+c}{1-c}.\]
\end{thm}
To motivate this result, we first point out that graphs with such an additive spectral gap are not necessarily expanders. Indeed, when $d$ is greater than $\Theta(\sqrt{n})$, the multiplicative spectral gap of graphs with a $\Theta(\sqrt{n})$ additive spectral graph will go to zero. Moreover, the diameter of graphs with such an additive spectral gap can still grow quite fast (polynomially in $n$), approaching the upper bound derived from Theorem~\ref{diameter_Alon}.

\subsection{Existence of graphs with large additive spectral gap and large diameter}\label{addi_gap_1}

\begin{prop}
For a regular graph with additive spectral gap $\delta$, the diameter $D$ is $O((n/\delta)^{1/3}(\log n)^{2/3})$.
\end{prop}
\begin{proof}
Let the regular graph have degree $d$.
From Theorem~\ref{diameter_Alon}, we know that
\begin{equation}\label{lem1}
D^2\le c\frac{d}{\delta}(\log n)^2
\end{equation}
where $c$ is some constant.
By Fact~\ref{regular_diameter_bound}, we also have
\begin{equation}\label{lem2}
D\le \frac{3n}{d}.
\end{equation}
 Multiplying~\eqref{lem1} and~\eqref{lem2}, we have
\begin{equation}
D\le (3c)^{1/3} \left(\frac{n}{\delta}\right)^{1/3}(\log n)^{2/3}.\qedhere
\end{equation}
\end{proof}

\begin{cor}
For a regular graph with an $\Omega(\sqrt{n})$ additive spectral gap, the diameter is $O(n^{1/6}(\log n)^{2/3})$.
\end{cor}

We show that this bound is nearly tight.

\begin{prop}
There are connected regular graphs with diameter $(1/2)n^{1/6}$ and an additive spectral gap of $cn^{1/2}$ where $c = 2\pi^2(1+O(n^{-1/3}))$.
\end{prop}

We prove a more general statement.

\begin{prop}\label{lexi_product}
For any constant $0<t<1$, there are connected regular graphs with diameter $(1/2)n^{(1-t)/3}$ and an additive spectral gap of $cn^t$ where\\ $c=2\pi^2(1+O(n^{(2t-2)/3}))$.
\end{prop}
\begin{proof}
Recall the lexicographic product and its properties in Definition~\ref{lex}, Observation~\ref{lex_obsv}, and Fact~\ref{lex_fact}. 
Let $G:= C_r \circ \overline{K}_s$ where $r\cdot s=n$. Then $G$ is a connected regular graph of degree $2s$ and diameter $\lfloor r/2  \rfloor$. The adjacency matrix of $G$ can be expressed as
\begin{equation}A_G=A_{C_r}\otimes J_s + I_r\otimes \vzr= A_{C_r}\otimes J_s.
\end{equation}
Since 
\begin{equation}\spec(J_s) = \{ \{ s, 0^{s-1} \} \},\end{equation} we have
\begin{equation}
\spec(G) = \{ \{ s\lambda_1(C_r),s \lambda_2(C_r),\dots, s\lambda_r(C_r), 0^{(s-1)r}. \} \}
\end{equation}
Thus \[\delta(G)=s(\lambda_1(C_r)-\lambda_2(C_r)).\] By Fact~\ref{spectrum_cycle},
\[\lambda_1(C_r)-\lambda_2(C_r)=2\left(1-\cos\left(\frac{2\pi}{r}\right)\right)
=\frac{4\pi^2}{r^2}+\epsilon\] 
where $\displaystyle{|\epsilon|<\frac{2^5\pi^4}{4! r^4}}$. We take $\displaystyle{r=n^{(1-t)/3}}$ and $\displaystyle{s=n^{(2+t)/3}}$. Then 
\[\delta(G) = 4\pi^2n^t\left(1+O\left(\frac{1}{r^2}\right)\right)= 4\pi^2n^t\left(1+O(n^{(2t-2)/3})\right).\qedhere\]
\end{proof}

\subsection{Large additive spectral gap implies bounded ratio}\label{addi_gap_2}\quad\\
\addtocontents{toc}{\protect\setcounter{tocdepth}{1}}

Now we prove Theorem~\ref{remote_2} which shows that $\gamma(G\pm e)$ is bounded when $G$ has an additive spectral gap of $(2+\epsilon)\sqrt{n}$ for some $\epsilon>0$. The proof is an adaptation of the theory developed in Chapter V. 2 in Stewart and Sun's book~\cite{StewartSun1990}, dealing with perturbation of invariant subspaces.  In addition to adapting the proofs to our special case, we fill in some details to make the proof accessible to readers not versed in perturbation theory. 

\begin{notation}
Let $U\in M_n(\R)$ be the orthogonal matrix whose columns are eigenvectors of $A$. We can write it as
\begin{equation}U=(\vx\;Y)=(\vx\;\vy_2\;\cdots\;\vy_n),\end{equation}
where the columns $\vx, \vy_2,\vy_3,\dots,\vy_n$ are eigenvectors of $A$ corresponding to eigenvalues $\lambda_1\ge\lambda_2\ge\lambda_3\ge\dots\ge\lambda_n$. Then $\lambda_1=d$, and we can assume \begin{equation}\vx=\left(\frac{1}{\sqrt{n}},\dots,\frac{1}{\sqrt{n}}\right)^T.\end{equation} In the context of adding an edge $e$ to $G$, we label the two endpoints of $e$ to be $1$ and $2$, and let $E\in M_n(\R)$ have $E_{21}=E_{12}=1$ while all other coordinates are zero. In the context of deleting an edge $e$ from $G$, we also label the two endpoints of $e$ to be $1$ and $2$, and let $E\in M_n(\R)$ have $E_{21}=E_{12}=-1$ while all other coordinates are zero. In both cases, $A+E$ is the adjacency matrix of the graph obtained.
\end{notation}

\subsubsection{Rotating the eigenbasis}
We know $\vx$ is not an eigenvector of $A+E$. We want to know how close $\vx$ is to the principal eigenvector of $G\pm e$, in the sense that we want to find a vector $\vv$ with small norm such that $\displaystyle{\widetilde{\vx}=\frac{\vx+\vv}{\|\vx+\vv\|}}$ is the unit principal eigenvector of $G\pm e$.

\begin{prop}\label{frac_use}
Let $\vp$ be a vector in $\R^{n-1}$. Let $U=(\vx\;Y)\in M_n(\R)$ be orthogonal, where $\vx$ is a vector in $\R^{n-1}$ and $Y$ is an $n\times (n-1)$ matrix.
Define $\widetilde{U}\in M_n(\R)$ as $\displaystyle{\widetilde{U}=(\widetilde{\vx}\;\widetilde{Y})}$, where 
\begin{equation}\label{remote}
    \widetilde{\vx}=\frac{\vx+Y\vp}{\sqrt{1+\|\vp\|^2}}\quad\text{and}\quad \widetilde{Y}=(Y-\vx\vp^T)(I_{n-1}+\vp\vp^T)^{-1/2}.
\end{equation}
Then $\widetilde{U}$ is an orthogonal matrix.
\end{prop}
\begin{proof}
\[\widetilde{\vx}^T\widetilde{\vx}=\frac{\vx^T\vx+\vp^TY^TY\vp}{1+\|\vp\|^2}=\frac{1+\|\vp\|^2}{1+\|\vp\|^2}=1.\]
\begin{align*}\widetilde{Y}^T\widetilde{Y}&=(I_{n-1}+\vp\vp^T)^{-1/2}(Y^T-\vp\vx^T)(Y-\vx\vp^T)(I_{n-1}+\vp\vp^T)^{-1/2}\\&=(I_{n-1}+\vp\vp^T)^{-1/2}(I_{n-1}+\vp\vp^T)(I_{n-1}+\vp\vp^T)^{-1/2}\\&=I_{n-1}.
\end{align*}
\begin{align*}
\widetilde{\vx}^T\widetilde{Y}&=
\frac{(\vx^T+\vp^TY^T)(Y-\vx\vp^T)(I_{n-1}+\vp\vp^T)^{-1/2}}{\sqrt{1+\|\vp\|^2}}\\
&=\frac{(-\vp^T+\vp^T)(I_{n-1}+\vp\vp^T)^{-1/2}}{\sqrt{1+\|\vp\|^2}}\\
&=0.
\end{align*}
Therefore \[\widetilde{U}^T\widetilde{U}=\begin{pmatrix} \widetilde{\vx}^T\widetilde{\vx} & \widetilde{\vx}^T\widetilde{Y}\\ \widetilde{Y}^T\widetilde{\vx} & \widetilde{Y}^T\widetilde{Y} \end{pmatrix}=I_n.\qedhere\]
\end{proof}

\subsubsection{Rotating $\vx$ to be an eigenvector}

We want to find $\vp$ such that $\widetilde{\vx}$ is the principal eigenvector of $A+E$.

\begin{notation}
We define \[e_{11}:=\vx^TE\vx\in\R,\quad e_{21}:=Y^TE\vx\in\R^{n-1},\quad E_{22}:=Y^TEY\in M_{n-1}(\R),\]
\[\text{and}\quad L:=Y^TAY=\degag(\lambda_2,\dots,\lambda_n)\in M_{n-1}(\R).\]
\end{notation}

\begin{obsv}\label{local_idk}
Since $\|E\|=1$ and $(\vx\:Y)$ is orthogonal, $\displaystyle{|e_{11}|, \|e_{21}\|, \|E_{22}\|\le1}$. 
\end{obsv}

\begin{prop}
If \begin{equation}\label{local}((d+e_{11})I_{n-1}-(L+E_{22}))\vp=e_{21}-\vp e_{21}^T\vp,\end{equation}
then $\widetilde{\vx}$ is an eigenvector of $A+E$.
\end{prop}
\begin{proof}
\eqref{local} is equivalent to
\[(Y^T-\vp\vx^T)(A+E)(\vx+Y\vp)=0,\] which gives \[\widetilde{Y}^T(A+E)\widetilde{\vx}=0.\] Since $(\widetilde{\vx}\;\widetilde{Y})$ is an orthogonal matrix, $\widetilde{\vx}$ is an eigenvector of $A+E$.
\end{proof}

\subsubsection{Finding small $\vp$ when the additive spectral gap is large}
\begin{notation}
We define $M\in M_{n-1}(\R)$ as $M=(d+e_{11})I_{n-1}-L-E_{22}$.
\end{notation}

\begin{prop}
If $\displaystyle{\delta>\frac{2}{c}\sqrt{n}+2}$ for some value $0<c<1$, then $M$ is non-singular and \begin{equation}
    \|M^{-1}\|\le\frac{1}{\delta-2}<\frac{c}{2\sqrt{n}}.
\end{equation}
\end{prop}
\begin{proof}
Since $(d+e_{11})I_{n-1}-L$ is a diagonal matrix and $E_{22}$ is a symmetric matrix, $M$ is symmetrix. Recall that $\|E_{22}\|\le 1$. By Rayleigh's principle and Observation~\ref{local_idk},
\begin{align*}
\min \spec(M) &= \min_{\|\vx\|=1} \vx^T ((d+e_{11})I_{n-1} - L - E_{22})\vx\\
&\ge \min_{\|\vx\|=1}\vx^T((d+e_{11})I_{n-1}-L)\vx  -1\\
&= \min_{\|\vx\|=1} \spec((d+e_{11})I-L) -1\\
&= d+e_{11}-\lambda_2 -1 \\
&\ge \delta -2.
\end{align*}

Therefore all eigenvalues of $M$ are positive, and 
\[
    \|M^{-1}\|=\max\spec(M^{-1})=\frac{1}{\min\;\spec(M) }\le\frac{1}{\delta-2}<\frac{c}{2\sqrt{n}}.\qedhere
\]
\end{proof}

\begin{prop}\label{local_4}
We write \eqref{local} in terms of $M$: \begin{equation}\label{local_2}M\vp=e_{21}-\vp e_{21}^T\vp.\end{equation}
Let $\theta=\|M^{-1}\|^{-1}$ and $\eta=\|e_{21}\|$. We claim that if $\delta>\frac{2}{c}\sqrt{n}+2$ for some value $0<c<1$, then
there exists $\vp$ with
\begin{equation}\label{local_3}
    \|\vp\|<\frac{2\eta}{\theta+\sqrt{\theta^2-4\eta^2}}
\end{equation} such that \eqref{local_2} holds, and thus $\widetilde{\vx}$ defined by \eqref{remote} is an eigenvector of $A+E$.
\end{prop}
\begin{proof}
We want to find a solution with small norm to the non-linear equation~\eqref{local_2}. We do this by an iterative construction.

We define a sequence of vectors $\vp_0,\vp_1,\dots$ such that
\[\vp_0=0\quad\text{and}\quad \vp_i=M^{-1}(e_{21}-\vp_{i-1}e_{21}^T\vp_{i-1}),\quad\text{for}\quad i\ge 1.\]

Then \[\|\vp_i\|\le  \|M^{-1}\|(|e_{21}|+\|\vp_{i-1}\|^2|e_{21}|)\le\frac{\eta(1+\|\vp_{i-1}\|^2)}{\theta}.\] 
We claim that the sequence $\{\vp_i\}$ converges.
We define
\[\xi_0=0,\quad\xi_i=\frac{\eta(1+\xi_{i-1}^2)}{\theta},\quad\text{for}\quad i\ge 1.\]

Then $\|\vp_i\|\le \xi_i$. Since $\xi_1=\frac{\eta}{\theta}>\xi_0$, we can prove by induction that $\xi_0,\xi_1,\xi_2,\dots$ is monotone increasing. Let
\[\phi(\xi)=\frac{\eta(1+\xi^2)}{\theta}.\] This function is monotone increasing in $\xi$, and has a fixed point at $\displaystyle{\xi=\frac{2\eta}{\theta+\sqrt{\theta^2-4\eta^2}}}$. Then $\xi_i<\xi_{i+1}=\phi(\xi_i)<\phi(\xi)=\xi$. Therefore the sequence $\{\xi_i\}$ converges to $\xi$. 
Thus 
\[\|\vp_i\|\le\xi_i\le\xi=\frac{2\eta}{\theta+\sqrt{\theta^2-4\eta^2}}<\frac{2\eta}{\theta}.\]
Next we prove the convergence of $\{\vp_0,\vp_1,\dots\}$. For any $i\ge 2$,
\begin{align*}
\|\vp_i-\vp_{i-1}\|&=|M^{-1}(\vp_{i-1}e_{21}^T\vp_{i-1}-\vp_{i-2}e_{21}^T\vp_{i-2})\\
&=|M^{-1}(\vp_{i-1}+\vp_{i-2})e_{21}^T(\vp_{i-1}-\vp_{i-2})|\\
&\le 2\|M^{-1}\|\|\vp_{i-1}\|\|\vp_{i-1}\|\cdot\eta\cdot\|\vp_{i-1}-\vp_{i-2}\|\\
&\le \frac{4\eta^2}{\theta^2}\|\vp_{i-1}-\vp_{i-2}\|.
\end{align*}
Then $\|\vp_i-\vp_0\|\le \rho^i\|\vp_1-\vp_0\|$,
where $\rho=\displaystyle{\frac{4\eta^2}{\theta^2}}<\frac{c^2}{4n}<1$. Therefore $\{\vp_i\}$ is a Cauchy sequence in $\R^{n-1}$ and has a limit $\vp$. Thus a solution $\vp$ exists, with norm satisfying \eqref{local_3}.

\end{proof}

\subsubsection{Using the bound on $\|\vp\|$ to show that $\tilde{\vx}$ is the principal eigenvector}

\begin{prop}
If $\delta>\frac{2}{c}\sqrt{n}+2$ for some value $0<c<1$, then the principal eigenvector $\widetilde{\vx}$ of $A+E$ can be writen in the form 
\[\widetilde{\vx}=\frac{\vx+Y\vp}{\sqrt{1+\|\vp\|^2}}\quad\text{where}\quad\|\vp\|<\frac{c}{\sqrt{n}}.\]
\end{prop}
\begin{proof}
We showed that there exists $\vp$ with \[\|\vp\|<\frac{2\eta}{\theta+\sqrt{\theta^2-4\eta^2}}<\frac{2\eta}{\theta}<\frac{c}{\sqrt{n}}\] such that $\displaystyle{\frac{\vx+Y\vp}{\sqrt{1+\|\vp\|^2}}}$ is an eigenvector of $A+E$. It remains to show that this is the principal eigenvector. 
Since $\|Y\|=1$ and $\|\vp\|<\displaystyle{\frac{c}{\sqrt{n}}}$ where  $0<c<1$, and since $\vx=\displaystyle{\left(\frac{1}{\sqrt{n}},\dots, \frac{1}{\sqrt{n}}\right)}$, all coordinates of $\widetilde{\vx}$ are positive. Therefore by Fact~\ref{positive_then_principal},  $\widetilde{\vx}$ is the principal eigenvector of $A+E$.
\end{proof}

\subsubsection{Completing the proof of Theorem~\ref{remote_2}}\quad\\
Following the argument above, we know the smallest possible coordinate of $\widetilde{\vx}$ is $\displaystyle{\frac{1-c}{\sqrt{n}}}$, and the largest possible coordinate of $\widetilde{\vx}$ is $\displaystyle{\frac{1+c}{\sqrt{n}}}$. Therefore the ratio of $G\pm e$ is as claimed. This completes the proof.\hfill\qed

\section{Open questions}

In Section~\ref{bounded_remove}, we showed that if we remove a non-bridge edge $e$ from a connected regular graph such that the endpoints of $e$ are of bounded distance in $G-e$, then $\gamma(G-e)$ is polynomially bounded. Is there also a polynomial bound when the endpoints of $e$ are of unbounded distance in $G-e$?
\begin{question}
If we remove a non-bridge edge $e$ from a connected regular graph $G$, is $\gamma(G-e)$ always polynomially bounded?
\end{question}

In Section~\ref{bounded_add}, we showed that  if we add an edge $e$ between two vertices at distance $2$ to a connected regular graph $G$ of bounded degree, then $\gamma(G+e)$ can be exponential. Can $\gamma(G+e)$ be exponential when $e$ is between two vertices of unbounded distance in $G$?
\begin{question}
If we add an edge $e$ to a connected regular graph $G$ with bounded degree $d$, such that the disance between the endpoints of $e$ in $G$ is unbouneded, can $\gamma(G+e)$ be exponential in $n$?
\end{question}

In Section~\ref{addi_gap_2}, we showed that for a connected regular graph $G$, an additive spectral gap greater than $2\sqrt{n}$ implies that $\gamma(G\pm e)$ is bounded (Theorem~\ref{remote_2}). However, this bound ceases to work at all for $G$ with an additive spectral gap $\delta\le 2\sqrt{n}$. Is this a limitaton of the method, or is there really an abrupt change at $\delta = 2\sqrt{n}$?
\begin{question}
Is there a family of connected regular graphs $G$ with an additive spectral gap slightly less than $2\sqrt{n}$ and with $\gamma(G\pm e)$ not bounded?
\end{question}

\section*{Acknowledgments}

I am very grateful to Prof.  Babai for his guidance and patience,  and to Prof.  Peter May for organizing the UChicago Math REU.


\begin{thebibliography}{9}

\bibitem{Fiedler1973}
Miroslav Fiedler. Algebraic connectivity of graphs. \textit{Czechoslovak Math. J.} (1973), 23(2):298-305. \href{http://doi.org/10.21136/CMJ.1973.101168}{\underline{DOI}}

\bibitem{AlonMilman1985}
Noga Alon, Vitali D. Milman.
$\lambda_1$, isoperimetric inequalities for graphs, and superconcentrators. \textit{J. Combinatorial Theory, Series B} (1985), 38:73-88. \href{https://doi.org/10.1016/0095-8956(85)90092-9}{\underline{DOI}}

\bibitem{StewartSun1990}
Gilbert W. Stewart, Ji-guang Sun. \textit{Matrix Perturbation Theory}, Academic press, 1990.

\bibitem{vanDamHaemers1995}
Edwin R. van Dam, Willem H. Haemers.
Eigenvalues and the Diameter of Graphs. \textit{Linear and Multilinear Algebra} (1995), 39:33-44. \href{https://doi.org/10.1080/03081089508818378}{\underline{DOI}}

\bibitem{Ng_ea2001}
Andrew Y. Ng, Alice X. Zheng, Michael I. Jordan. Link analysis, eigenvectors and stability.  In: \emph{17th Internat. Joint
Conf. on A. I.  (IJCAI'01)},  vol. 2,  pp. 903-910, Morgan Kaufmann,  2001.
\href{http://www.robotics.stanford.edu/~ang/papers/ijcai01-linkanalysis.pdf}{\underline{Author's website}}

\bibitem{CioabaGregory2007}
Sebastian M. Cioab\u{a}, David A. Gregory. Principal eigenvectors of irregular graphs. \textit{Electr. J. Linear Algebra} (2007), 16:366-379. \href{https://doi.org/10.13001/1081-3810.1208 }{\underline{DOI}}

\bibitem{CioabaGregoryNikiforov2007}
Sebastian M. Cioab\u{a}, David A. Gregory, Vladimir Nikiforov. Extreme eigenvalues of nonregular graphs. \textit{J. Combinatorial Theory, Series B} (2007), 97(3):483-486. \href{https://doi.org/10.1016/j.jctb.2006.07.006}{\underline{DOI}}

\bibitem{TaitTobin2018}
Michael Tait, Josh Tobin. Characterizing graphs of maximum principal ratio. \textit{Electr. J. Linear Algebra} (2018),  34:61-70. \href{https://doi.org/10.13001/1081-3810.3200 }{\underline{DOI}}

\end{thebibliography}
\end{document}